\documentclass[10pt]{amsart}
\usepackage{amssymb, amsmath}
\usepackage{graphics,xcolor}
\usepackage{latexsym,amsmath,amssymb,amsthm,amsfonts}
\usepackage{amscd}
\usepackage[arrow, matrix, curve]{xy}

\DeclareMathOperator{\gal}{Gal}

\DeclareMathOperator{\gl}{GL}

\DeclareMathOperator{\mon}{Mon}
\DeclareMathOperator{\sym}{Sym}
\DeclareMathOperator{\mt}{MT}
\DeclareMathOperator{\su}{SU}
\DeclareMathOperator{\sh}{Sh}
\DeclareMathOperator{\coker}{Coker}
\DeclareMathOperator{\mult}{mult}
\DeclareMathOperator{\psu}{PSU}
\DeclareMathOperator{\psp}{Psp}
\DeclareMathOperator{\gsp}{Gsp}
\theoremstyle{plain}
\newtheorem{thm}{Theorem}[section]
\newtheorem{theorem}[thm]{Theorem}
\newtheorem{lemma}[thm]{Lemma}

\newtheorem{proposition}[thm]{Proposition}

\theoremstyle{definition}

\newtheorem{remark}[thm]{Remark}

\newtheorem{definition}[thm]{Definition}

\newtheorem{example}[thm]{Example}

\numberwithin{equation}{thm}
\newtheorem{constr}[thm]{Construction}

\newcommand{\C}{{\mathbb C}}

\renewcommand{\P}{{\mathbb P}}
\newcommand{\Q}{{\mathbb Q}}

\newcommand{\V}{{\mathbb V}}

\newcommand{\Z}{{\mathbb Z}}

\newcommand{\Hom}{{\rm Hom}}

\makeatletter
\def\blfootnote{\gdef\@thefnmark{}\@footnotetext}
\makeatother

\begin{document}

\title{Shimura varieties and abelian covers of the line}
\author{Abolfazl Mohajer}
\address{Universit\"{a}t Mainz, Fachbereich 08, Institut f\"ur Mathematik, 55099 Mainz, Germany}
\email{mohajer@uni-mainz.de}
\blfootnote{\textup{2000} \textit{Mathematics Subject Classification}:
	11G15, 14G35, 14H40}
\keywords{Jacobian variety, Shimura variety, special subvariety}
\maketitle

\begin{abstract}

We prove that under some conditions on the monodromy, families of abelian covers of the
projective line do not give rise to (higher dimensional) Shimura subvarieties in $A_g$. This is achieved by
a reduction to $p$ argument. We also use another method based on monodromy computations to show that two dimensional 
subvarieties in the above locus are not special. In particular it is shown that such families have usually large monodromy groups. Together
with our earlier results, the above mentioned results contribute to classifying the special families in the moduli space of abelian
varieties and partially completes the work of several authors including the author's previous work.
\end{abstract}

\section{introduction} \label{intro}

In this paper we continue the study started in \cite{MZ} of
Shimura subvarieties generated by families of abelian covers $C\to
\mathbb{P}^{1}$. More precisely, let $A_{g}$ the moduli space of
principally polarized complex abelian varieties of dimension $g$
and $M_{g}$ be the (coarse) moduli space of non-singular complex
algebraic curves of genus $g$. The period mapping $j:M_g\to A_g$
is well-known to be an immersion away from the hyperelliptic
locus. We denote the image of this map by $T_g^{\circ}$. Its
closure is denoted by $T_g$. By a conjecture of Coleman, it is expected that
for large $g$, there is no positive dimensional special (Shimura)
subvariety $Z\subset T_g$ with $Z\cap T_g^{\circ}\neq \emptyset$, see \cite{C}. This problem
has been considered by many authors. If $g\leq 7$, then 
there are counter-examples going back to Shimura \cite{S} and more recently \cite{DN}, 
\cite{FGP}, \cite{FPP}, \cite{M} and \cite{MO}. A special Subvariety of $A_g$ is a totally geodesic subvariety,
hence the above expectation that there is no such subvariety in a
rather curved space such as $T_g$, see \cite{CFG}, where also the
fundamental form of the period map is investigated to give an
upper bound for the possible dimension of a totally geodesic
submanifold of $A_g$. In this paper, we consider subvarieties 
$Z$ of $A_g$ arising from families of abelian covers of $\mathbb{P}^{1}$. In this case, 
the subvariety $Z$ will be of dimension $s-3$, where $s$ is the (fixed) number of branch 
points of the covers, see section 2. In \cite{MZ}, we have proved various results
in this direction. More explicitly, we have classified all
1-dimensional families of irreducible abelian coverings of the
line. This work is a continuation in the line of the previous works of
several authors, for example \cite{CFG}, \cite{DN}, \cite{FGP}, \cite{FPP}, \cite{M} and \cite{MO}. There is a special subvariety $S(G)$, 
constructed using endomorphims of Jacobians induced by the action of $G$, containing $Z$.  
If $\dim Z=\dim S(G)=s-3$, then since $Z$ is irreducible, it follows that $Z=S(G)$ and hence $Z$ is special. 
However, if $Z\neq S(G)$, there is no reason that $Z$ is not a special subvariety. It could still be that it is
a special subvariety of smaller dimension. Our first main result, Theorem~\ref{mainthm1}, amounts to say that
under some conditions on the monodromy, the special subvarieties are exactly those that satisfy
the above equality. To prove the above theorem, we have used a reduction to $p$ technique due to Dwork and Ogus; 
these techniques were used in \cite{DN} by de Jong and Noot in the setting of Shimura subvarieties in the Torelli locus; still later these ideas 
were used in \cite{M} to give a complete answer in the case of cyclic covers. We remark that one can develop the same ideas to
get special families of covers of other curves, e.g., elliptic curves,
rather than the projective line, see \cite{FPP}. Note that in \cite{MZ}, the families which satisfy the above equality are listed
for a bounded number of $N,m$ and $s$. This table appears also in \cite{MO} (and \cite{FGP}) and is obtained by a slightly different method. Moreover, in \cite{MZ} we have shown that
when the Torelli image $Z$ is of large dimension (i.e., $s$ is large enough), the family of abelian covers never gives
rise to a  special subvariety in $A_g$. This was done, by computing
the dimension of eigenspaces of the cohomology and showing that
these spaces constitute a large monodromy group. We push this
method further in the present article and our second main result, Theorem~\ref{mainthm2}, shows that, except for two possible exceptions,
there are no 2-dimensional special subvarieties arising from families of abelian coverings in $A_g$. The proof shows however, that as in the case
of large-dimensional subvarieties treated in \cite{MZ}, our method
can be also applied for other low-dimensional subvarieties. 

\section*{Acknowledgement} The author would like to thank the referee for useful remarks and suggestions
that increased the clarity and overall readability of the paper.

\section{preliminaries}

An abelian Galois cover of $\mathbb{P}^{1}$ is determined by a
collection of equations in the following way: Take an $m\times
s$ matrix $A=(r_{ij})$ with entries in
$\mathbb{Z}/N\mathbb{Z}$ for some $N\geq 2$. Let
$\overline{\mathbb{C}(z)}$ be the algebraic closure of
$\mathbb{C}(z)$. For each $i=1,\cdots,m$ choose a function $w_{i}\in
\overline{\mathbb{C}(z)}$ with
\[w_{i}^{N}=\prod_{j=1}^{s}(z-z_{j})^{\widetilde{r}_{ij}} \text{ for  } i=
1,\cdots, m, \]
in $\mathbb{C}(z)[w_{1},...,w_{m}]$ with $z_j\in
\mathbb{C}$ and $\widetilde{r}_{ij}$ the lift of $r_{ij}$ to $\mathbb{Z}\cap [0,N)$. We assume that the sum of columns are zero (in $\Z/N$), so that
the the cover os not branched over the infinity. There exists a projective non-singular curve $Y$
birational to the affine curve defined by the above equations together with a 
covering map $\pi : Y\to \mathbb{P}^{1}$ with abelian deck transformation group. The group $G$ can be realized as 
the column span of the matrix $A$. The local monodromy around the branch point $z_{j}$ is
given by the column vector $(r_{1j},....,r_{mj})^{t}$ and hence the
order of ramification over $z_{j}$ is $
\frac{N}{\gcd(N,\widetilde{r}_{1j},..,\widetilde{r}_{mj})}$. This fact combined with the Riemann-Hurwitz formula gives that the genus $g$ of the cover
can be computed by the following formula
\[g= 1+ d(\frac{s-2}{2}- \frac{1}{2N}\sum_{j=1}^{s}
\gcd(N,\widetilde{r}_{1j},...,\widetilde{r}_{mj})),\] 
where $d$ is the degree of the covering which is equal, as pointed out above,
to the column span (equivalently row span)  of the matrix $A$. In
this way, the Galois group $G$ of the covering will be a subgroup
of $(\mathbb{Z}/N\mathbb{Z})^{m}$.

\begin{remark} \label{rowspan}

Consider two families of abelian covers with
matrices $A$ and $A^{\prime}$ over the same
$\mathbb{Z}/N\mathbb{Z}$. If the row spans of $A$ and $A^{\prime}$ are equal,
then the two families are isomorphic. See \cite{MZ}.
\end{remark}

\begin{example} \label{explicitabelian}
 We provide an explicit example of an abelian cover of $\P^1$ using the datum $(N,s,A)$.
 Consider the triple $(3,5,A=\begin{pmatrix} 2&1&2&1&0\\
0&0&1&1&1 \end{pmatrix})$ and fix $z_1,\cdots, z_5\in \C$ (these are the branch points, for example take $z_1,\cdots, z_5=1,\cdots, 5$) . Note that this is one of the exceptional families
(family $III^*$) that appear in section 5, Theorem~\ref{mainthm2}. From the matrix one can read that the cover is given by two equations

\begin{align}
w_1^3&=(z-z_1)^2(z-z_2)(z-z_3)^2(z-z_4) \nonumber\\
w_2^3&=(z-z_3)(z-z_4)(z-z_5) \nonumber
\end{align}

The Galois group $G$ is equal to the column span of the marix $A$ 
which is equal to $\Z/3\Z \times \Z/3\Z$. This can be seen as follows and the proof can be easily generalized to 
an arbitrary abelian covering. Let $l_1,\cdots, l_5$ be the five columns of the matrix $A$. Consider the field 
$E=\C(z)[u_1,\cdots, u_5]$, with $u_j$ a third root of $z-z_j$. Note that if $F$ is the fraction field of the covering,
then $F\subseteq E$. Now the set of products $\displaystyle \prod_{j=1}^{5}u_j^{c_j}$ form a basis for $E$ over $\C(z)$ 
for $0\leq c_j<3$. The Galois group of $E$ over $\C(z)$ is $(\Z/3\Z)^4$ with generators $\psi_j(u_j)=\xi_3u_j$ for $\xi_3$ 
a primitive third root of unity. If we take the products $\displaystyle \prod_{j=1}^{5}u_j^{\widetilde{r}_{j}}$ for $(r_1,\cdots, r_5)$
in the row span of $A$, then we get a basis for $F$ over $\C(z)$. So, in particular, the degree of the covering is also the 
size of the row span of $A$. Note that the elements of the Galois group of $F$ over $\C(z)$ are given by the restriction of elements from Galois group of $E$ over $\C(z)$
to $F$ (in other words the homomorphism $\gal(E/\C(z))\to \gal(F/\C(z))$ given by the restriction of automorphisms is surjective). 
Given an element $\psi=\displaystyle \prod_{j=1}^{5}\psi_j^{p_j}$, we compute its action on $w_1$ and $w_2$ (i.e., its restriction to $F$).
$\psi$ acts on $w_i$ by multiplication with $\displaystyle \xi_3^{\sum_{j=1}^{5}r_{ij}p_j}$. Thus the column vector $(c_1,c_2)^t=p_j\cdotp l_j$ gives the action of $\psi$ on $w_i$
which is by multiplication with $\xi_3^{c_i}$. This shows that the Galois group is isomorphic to the column span of $A$. Furthermore, each column vector
$l_j$ records the local monodromy around $z_j$ and hence the order of ramification around $z_j$ is 3. Using the Riemann-Hurwitz formula one sees that the genus
of the covering is thus equal to $7$.

\end{example}

Let $T \subset (\mathbb{A}^{1})^{s}$ be the complement of the big
diagonals, i.e., $T=\mathcal{P}_{s}= \{(z_{1},....,z_{s})\in
(\mathbb{A}^{1})^{s}\mid z_{i}\neq z_{j} \text{   }\forall i\neq j \}$. Over
this open affine set, we define a family of abelian covers of
$\mathbb{P}^{1}$ to have the equation:
\[w_{i}^{N}=\prod_{j=1}^{s}(z-z_{j})^{\widetilde{r}_{ij}} \text{ for  }i=
1,\cdots, m,\]
where the tuple $(z_{1},...,z_{s})\in T$ and $\widetilde{r}_{ij}$
is the lift of $r_{ij}$ to $\mathbb{Z}\cap [0,N)$ as before. 
Varying the branch points we get a family $f:C\to T$ of
smooth projective curves over $T$ whose fibers $C_t$ are abelian
covers of $\mathbb{P}^{1}$ introduced above. If $f:C\rightarrow T$ is a family of abelian covers constructed as
above, we write $J\rightarrow T$ for the relative Jacobian of $C$
over $T$. This family gives a natural map $j:T\rightarrow A_{g}$.
Let $Z= Z(N,s,A)$ be the closure $\overline{j(T)}$ in $A_{g}$.
Hence $Z=Z(N,s,A)$ is a closed algebraic subvariety in $A_{g}$ and we have $\dim Z=s-3$, see \cite{FGP}, \S 2.3 for a 
detailed discussion on this. We call the subvariety $Z$ the \emph{moduli variety} associated to the family $f:C\rightarrow T$.
As $A_g$ has the structure of a Shimura variety, it makes sense to talk
about its special (or Shimura) subvarieties. Our goal is to classify the cases where $Z$ is a Special subvariety. Such families 
have a dense set of CM fibers. On the other hand, if $Z$ is not special then the Andr\'e-Oort theorem for $A_g$, 
see \cite{T}, implies that the set of CM fibers is not dense in $Z$. Let $G$ be a finite
abelian group. We denote by $\mu_{G}$ the group of the characters
of $G$, i.e., $\mu_{G}=\Hom(G,\mathbb{C}^{*})$. Consider a Galois
covering $\pi: X\rightarrow \mathbb{P}^{1}$ with Galois group $G$ constructed at the beginning of this section.
The group $G$ acts on the sheaves $\pi_{*}(\mathcal{O}_X)$, $\pi_{*}(\omega_X)$
($\omega_X$ is the canonical sheaf of $X$) and
$\pi_{*}(\mathbb{C})$ via its characters and we get corresponding
eigenspace decompositions $\pi_{*}(\mathcal{O})=\oplus_{\chi \in
	\mu_{G}} \pi_{*}(\mathcal{O})_{\chi}$ and
$\pi_{*}(\mathbb{C})=\oplus_{\chi \in \mu_{G}}
\pi_{*}(\mathbb{C})_{\chi}$. 

\begin{remark} \label{abeliangroupcharacter}
If $G$ is a finite abelian group,
then $\mu_G= \Hom(G,\mathbb{C}^{*})$ is isomorphic to $G$. To see this, write $G=\Z_{r_1}\times\cdots \times \Z_{r_m}$ as a product of finite cyclic groups. 
There is an isomorphism $\phi_G\colon G\to \mu_G$ given by sending the element $g\in G$ to the character 
$\chi_g:G\to\C^{*}, h\mapsto \exp(2\pi igD^{-1}h^t)$, where $D=diag(r_1,\cdots,r_m)$ and elements of $G$ are regarded as $1\times m$ matrices.
\end{remark}

From now on, we fix an isomorphism of $G$ with a product of finite cyclic groups and
an embedding of $G$ into $(\mathbb{Z}/N)^{m}$, where $m$, as in the beginning of this section, is the number of rows of the matrix of the covering. By the above isomorphism, we identify a 
character $\chi$ with the corresponding group element $n\in G$. Let $l_{j}$ be the $j$-th column of the matrix $A$. As mentioned
earlier, the group $G$ can be realized as the column span of the
matrix $A$. Therefore we may assume that $l_{j}\in G$. For a
character $\chi$, $\chi(l_{j})\in \mathbb{C}^{*}$ and since $G$ is
finite $\chi(l_{j})$ will be a root of unity. Let
$\chi(l_{j})=e^{\frac{2\alpha_{j}\pi i}{N}}$, where $\alpha_{j}$
is the unique integer in $[0,N)$ with this property. 

\begin{remark} \label{eigenspacedim}
Equivalently, the $\alpha_{j}$ can be obtained in the following way: Note that the abelian Galois
group $G$ of the covering is a (possibly proper) subgroup of
$\mathbb{Z}^{m}_{N}$ and therefore we can denote an element of $G$
by an $m$-tuple $n=(n_{1},...,n_{m})$. We regard $n$ as an $1\times
m$ matrix. Then the matrix product $n\cdotp A$ is meaningful and we set
$n\cdotp A=(\alpha_{1},...,\alpha_{s})$. We call this tuple, \emph{the associated tuple} to the eigenspace. Here all of the operations are
carried out in $\mathbb{Z}/N$ but the $\alpha_{j}$ are regarded as
integers in $[0,N)$. The space $H^{1,0}_{n}$ of
differential forms on which $G$ acts via character $n$ is
generated over $\mathbb{C}(z)$ by the differential form
$\omega_n=\prod(z-z_{j})^{-t_{j}(-n)}dz$, where
$t_{j}(n)=\langle\frac{\sum_{i=1}^{m} n_{i}\widetilde{r}_{ij}}{N}\rangle $ and
$\langle \cdot \rangle$ denotes the fractional part of a real number. Set $t(n)= \sum t_{j}(-n)$. It is then
straightforward to check that the forms $\omega_{n,\nu}=z^{\nu} \prod(z-z_{j})^{-t_{j}(-n)}dz$ constitute a $\mathbb{C}$-basis for $H^{1,0}_{n}$
where $0\leq \nu \leq t(n)-2$. See \cite{W}, Lemma 2.6. So $d_n=\dim H^{1,0}_{n}$ is equal to
$t(n)-1$.
\end{remark}

Summarizing the above, $d_n$ can be computed by the below formula.

\begin{proposition} [\cite{MZ}, Prop. 2.8] \label{dncomputation} 
For $C$ an abelian cover of $\mathbb{P}^{1}$ we have: 
\[d_{n}=h^{1,0}_{\chi}(C)=-1+\sum_{1}^{s}\langle - \frac{\alpha_{j}}{N}\rangle.\]
\end{proposition}

\begin{example}
We provide an explicit example of the computation of the $d_n$ using Remark~\ref{eigenspacedim}. Consider
the family in Example~\ref{explicitabelian} and let $n=(1,1)\in \Z/3\Z \times \Z/3\Z$ and $(1,1).A=(\alpha_1,\cdots, \alpha_5)=(2,1,0,2,1)$. 
Let 
\[\omega=\displaystyle \prod_{j=1}^{5}(z-z_j)^{-\frac{\alpha_j}{3}}dz=(z-z_1)^{-\frac{2}{3}}(z-z_2)^{-\frac{1}{3}}(z-z_4)^{-\frac{2}{3}}(z-z_5)^{-\frac{1}{3}}dz.\]
Recall the operators $\psi_j$ of Example~\ref{explicitabelian}. The action of the $\psi_j$ on the function field $F$ of the cover is linear 
and $\displaystyle \prod_{j=1}^{5}(z-z_j)^{-\frac{\alpha_j}{3}}$ spans the simultaneous eigenspace where $\psi_j$ acts by multiplication by $\xi^{\alpha_j}_3$. 
Thus any form in $H^{1,0}_{n}$ is of the form $p(z)\omega$ with $p(z)\in \C(z)$.
One can show that a meromorphic one-form $\eta=p(z)\omega$ is holomorphic if and only if $p(z)$ is a constant. 
Indeed, let $p(z)=p_0(z)\displaystyle \prod_{j=1}^{5}(z-z_j)^{s_j}$ where $p_0(z)$ has no zeros or poles in $z_j$. The order of ramification
of the cover over each $z_j$ is 3 so locally $u^3=(z-z_j)$ and $\eta$ is proportional to $u^{3s_j-\alpha_j+2}du$ near a lift of $z_j$. This is hlomorphic at lifts of $z_j$ iff
$s_j\geq \frac{-2+\alpha_j}{3}$. As $s_j$ is an integer, this implies that $s_j\geq 0$. In particular $p(z)$ is a polynomial, say of degree $d$. The condition that $\eta$
is holomorphic at lifts of $\infty$ gives that $d\leq 0$ (by using the chart $z\mapsto \frac{1}{z}$) so $d=0$. 
\end{example}

Consider the moduli variety $Z$ of the family as in the introduction. There are two special subvarieties containing $Z$ denoted by $S(G)$ and $S_f$. The  special subvariety $S(G)$ 
is defined by the action of the group $G$. This subvariety also appears in 
\cite{FGP} \S 3.2 ($Z(D_G)$ in their notation). It is constructed using endomorphims of Jacobians induced by the action of $G$. For more details we refer to \cite{M},
which is probably the first work to introduce these ideas, see also \cite{MZ}. 
In particular, if $\dim Z=\dim S(G)=s-3$, then since $Z$ is irreducible, it follows that $Z=S(G)$ and hence $Z$ is special. Therefore if we can compute
the dimension of $S(G)$, this gives us a numerical condition for $Z$ to be special. This is so far the main criterion for detecting if $Z$ is special. In fact our main theorems
in the section 5 amount to say that the special subvarieties are exactly those satisfying this condition. Note that in \cite{MZ}, the families which satisfy this condition are listed
for a bounded number of $N,m$ and $s$. Dimension of $S(G)$ can be computed by the following lemma.
\begin{lemma} [\cite{MZ}, Lemma 2.6] \label{dimS(G)}
$\dim S(G)= \sum_{2n\neq 0} d_{n}d_{-n}+ \frac{1}{2}\sum_{2n=0}d_{n}(d_{n}+1)$.
\end{lemma}

\begin{remark}
In \cite{MZ}, Table 1 some examples of special families of
abelian covers are listed. This table appears also in \cite{MO} and \cite{FGP}. The latter contains more examples of special
families of curves even of non-abelian Galois covers of the projective line. As explained above, similar conditions 
are formulated ($\dim Z(D_G)=s-3$) and it is checked, using a computer program, which families satisfy this numerical condition.
Hence all of the special families satisfy the equality $\dim Z=\dim S(G)=s-3$. 
However, if $\dim S(G)>s-3$ one can not a priori imply that the family is not a spceial family.
It could very well be that the family gives rise to a smaller special subvariety. In \cite{M}, it is shown
that for families of cyclic covers, this is not the case. In \cite{MZ} we have also shown that in the case of 
families of abelian covers over curves (i.e., with $s=4$ branch points) this actually does not happen. However this proof does not hold when $s>4$. 
We will show in Theorem~\ref{mainthm1}  that under some additional conditions, this can imply that the family is not special. 
\end{remark}

However, as the following example shows, it could happen that the Torelli image $Z$ is not a special
variety but it contains one.

\begin{example} Consider the family $f:C\rightarrow T$ of
cyclic covers of $\mathbb{P}^{1}$ of genus 7 given by the ramification data
$(12,(4,6,7,7))$. This is a Special family which appears as family (20) of \cite{MO}, Table 1. Take a $t \in T$. The fiber $C_{t}$ is a
cyclic Galois cover of $\mathbb{P}^{1}$. The Galois group of
$C_{t}\rightarrow \mathbb{P}^{1}$ is $\mathbb{Z}/12\mathbb{Z}$
which contains a normal subgroup $H\cong\mathbb{Z}/6\mathbb{Z}$.
Now this cover factors as
\[C_{t}\rightarrow X_{t}\rightarrow \mathbb{P}^{1}\]
The quotient cover $C_{t}\rightarrow X_{t}$ is a Galois cover with
Galois group $H\cong\mathbb{Z}/6\mathbb{Z}$. We have $X_{t}\cong
\mathbb{P}^{1}$ as the quotient corresponds to the equation
$y^{2}=(x-x_{1})(x-x_{2})$. The cover $X_{t}\rightarrow
\mathbb{P}^{1}$ has Galois group $\mathbb{Z}/2\mathbb{Z}$ and so
has degree $2$ and is branched above $2$ points. The cover $C_{t}\rightarrow X_{t}(\cong \mathbb{P}^{1})$ is
ramified at 6 points with ramification indices $1,1,2,2,3,3,$.
This shows that our original family  $f:C\rightarrow T$ is a
1-dimensional subfamily of the $3$-dimensional family of cyclic
covers given by the ramification data $(6,(1,1,2,2,3,3))$. It is
quite easy to see from the Moonen's list that this family is not a
Shimura family. Indeed, for this family $\dim S(G)=5\neq 3$,
so by the results of Moonen, this family is not Shimura. Note that the subfamily here is given over the closed subset (of
$(\mathbb{P}^{1})^{6}\setminus \widetilde{\Delta}$) given by a set
$\Gamma$ of orbits of the $\mathbb{Z}/2\mathbb{Z}$-action on
$\mathbb{P}^{1}$.\\ 

Therefore we have found a $3$-dimensional non-special family of
cyclic covers which contains a $1$-dimensional special subvariety.\\

Similarly one finds that:\\

$\bullet$ The family given by the ramification data
$(2,(1,1,1,1,1,1,1,1))$ (the universal family of hyperelliptic
curves of genus $g=3$) which is \emph{not} a Shimura family (cf.
\cite{M}) has a subfamily isomorphic to the family $(4,(1,1,2,2,2))$.
So in this example $Z$ contains a Shimura surface.

\

\

$\bullet$ The family $(4,(1,1,1,1,2,2))$ (a family of cyclic
covers of fiber genus $g=5$) has a subfamily isomorphic to the
family $(8,(5,5,4,2))$. In this case $\dim Z=3$ and $Z$ contains a
Shimura curve.

\end{example}

\section{The variety $S_f$}
In this section we construct a second special subvariety containing the moduli variety $Z$
associated to a given family $f\colon C\to T$ of abelian covers of $\P^1$. First we need to have 
some fundamental definitions and results to understand this construction. The definitons will be given 
in the more general setting of families of manifolds and their variations of Hodge structures and will
subsequently be applied to families of curves.

\begin{definition}
 Let $(V,h)$ be a $\Q$-Hodge structure given by a homomorphism $h\colon \mathbb{S}\to \gl(V)$. The \emph{Mumford-Tate group} of $(V,h)$ is the smallest $\Q$-algebraic subgroup of $\gl(V)$ containing $h(\mathbb{S})$. For more details
 on the representation theoretic aspect of Hodge structures, see \cite{D1}. 
\end{definition}

The following definition generalizes the above definition to families of Hodge structures over smooth non-singular algebraic varieties.

\begin{definition} \label{genericmum}
 Let $f\colon X\to S$ be a family of compact K\"ahler manifolds with $S$ a connected complex manifold and $\V=R^1f_*\Q$ 
 the associated polarized variation of $\Q$-Hodge structures. Then there is a dense subset $S^{\circ}$, such that
 the Mumford-Tate groups of all Hodge structures in $\V$ coincide, see for instance \cite{A} and \cite{D2}. In other words for all $s,s^{\prime}\in S^{\circ}$, 
 $\mt(\V_s)=\mt(\V_{s^{\prime}}):=M$. If $t\in S\setminus S^{\circ}$, then $\mt(\V_t)\subsetneq M$. We call the 
 group $M=\mt(\V_s)$ for $s\in S^{\circ}$, the \emph{generic Mumford-Tate group} of the family. The points 
 in $S^{\circ}$ are called the \emph{Hodge-generic points}. 
\end{definition}

We remark that it has been shown in \cite{CDK} that if $S$ is a non-singular algebraic variety, then $S\setminus S^{\circ}$ is a countable union of proper algebraic subvarieties.

For the following construction, note that if $G$ is an algebraic group, then $G^{ad}$ is the quotient of $G$ obtained
by the adjoint representation of $G$ on its Lie algebra $\mathfrak{g}$. If $G$ is a connected algebraic group defined over $\Q$, 
then the adjoint group $G^{ad}$ is isomorphic to the group $G/Z(G)$, where $Z(G)$ is the center of $G$.

\begin{constr} \label{smallestshimura} For the second special subvariety which contains $Z$, let $M$ be the generic Mumford-Tate group of the family $f:C\rightarrow T$. The group $M$ is a reductive $\mathbb{Q}$-algebraic group. Let $S_{f}$ be the natural
Shimura variety associated to $M$. In general
we have $S_{f}\subseteq S(G)$. The Shimura subvariety $S_{f}$
is in fact the \emph{smallest} special subvariety that contains
$Z$ and its dimension depends on the real adjoint group
$M^{ad}_{\mathbb{R}}$. Explicitly, if $M^{ad}_{\mathbb{R}}=
Q_{1}\times...\times Q_{r}$ is the decomposition of
$M^{ad}_{\mathbb{R}}$ to $\mathbb{R}$-simple groups then $\dim
S_{f}= \sum \delta(Q_{i})$. The numbers $\delta(Q_{i})$ are defined as follows: If $Q_{i}(\mathbb{R})$ is
not compact, $\delta(Q_{i})$ is the dimension  of the corresponding symmetric space associated to the real 
group $Q_{i}$ which can be read from Table V of \cite{H}. If $Q_{i}(\mathbb{R})$ is compact, i.e.,
if $Q_{i}$ is anisotropic we set $\delta(Q_{i})=0$. We remark that
for $Q=\psu(p,q)$, $\delta(Q)=pq$ and for $Q=\psp_{2p}$,
$\delta(Q)=\frac {p(p+1)}{2}$. Note also that $Z$ is a Shimura
subvariety if and only if $\sum\delta(Q_{i})=s-3$, i.e., if and
only if $\dim Z=\dim S_{f}=s-3$.
\end{constr}

The following remark, proved originally in \cite{A}, will be useful in the proofs of the main results.
\begin{remark} \label{monodromydecomp}
 If a family of curves $f:C\rightarrow T$ gives rise
 to a Shimura subvariety in $A_{g}$ then the connected monodromy
 group $\mon^{0}$ is a normal subgroup of the generic Mumford-Tate
 group $M$. In fact in this case $\mon^{0}=M^{der}$. Consequently, if $f$ gives rise to a Shimura subvariety and $M^{ad}_{\mathbb{R}}=\prod_{1}^{l} Q_{i}$ as a product of simple
 Lie groups then $\mon^{0,ad}_{\mathbb{R}}=\prod_{1}^{l} Q_{i}$.
 \end{remark}

\subsection{Monodromy of families of abelian covers}

Let $f:C\rightarrow T$ be a family of
 abelian Galois covers of $\mathbb{P}^{1}$ as constructed in
 section $2$. Then the local system
 $\mathcal{L}=R^{1}f_{*}\mathbb{C}$ gives rise to a polarized
 variation of Hodge structures (PVHS) of weight $1$. We have the monodromy representation $\pi_{1}(T,x)\to \gl(V)$, where
 $V$ is the fiber of $\mathcal{L}$ at $x$. The Zariski closure of the image of this morphism is
 called the \emph{monodromy group} of $\mathcal{L}$. Let $\mon^{0}(\mathcal{L})$ be the
 identity component of this group. The
 PVHS decomposes according to the action of the abelian Galois
 group $G$ and the eigenspaces $\mathcal{L}_{i}$ (or
 $\mathcal{L}_{\chi}$ where $i\in G$ corresponds to character $\chi
 \in \mu_{G}$ by Remark~\ref{abeliangroupcharacter}) are again variations of Hodge
 structures. For $t\in T$ suppose that $h^{1,0}((\mathcal{L}_{i})_{t})=a$ and
 $h^{0,1}((\mathcal{L}_{i})_{t})=b$. Since the monodromy group respects
 the polarization, $(\mathcal{L}_{i})_{t}$ is equipped with a Hermitian form of
 signature $(a,b)$ (see \cite{DM}, 2.21 and 2.23). This implies that
 $\mon^{0}(\mathcal{L}_{i}) \subseteq U(a,b)$. 
 
 \begin{definition} \label{differenttypes}
 With the notation as above, we say that $\mathcal{L}_{i}$ is \emph{of type} $(a,b)$. Two eigenspaces
 $\mathcal{L}_{i}$ and $\mathcal{L}_{j}$ of types $(a,b)$ and
 $(a^{\prime},b^{\prime})$ are said to be \emph{of distinct types} if $\{a,b\}\neq \{a^{\prime},b^{\prime}\}$.
  \end{definition}
 
 The above observations are key to our further analysis. Let us
 mention a lemma for whose proof we refer to \cite{MZ}, Prop 6.4 and 
 is a generalization of \cite{M}, Prop.4.7 to the setting of abelian covers.
 
 \begin{lemma} \label{monodromyeigspace}
Let $\mathcal{L}_{i}$ be an eigenspace as
 discussed above of type $(a,b)$ with $a,b\geq 1$. Then
 $\mon^{0}(\mathcal{L}_{i})=\su(a,b)$ unless $|G|=2l$ is even
 and $i$ is of order $2$ in $G$, in which case there is a
 surjection from $\mon^{0}(\mathcal{L}_{i})$ to $\su(n,n)=Sp_{2n}$,
 where $n=d_{i}$.
 \end{lemma}
 
 \begin{example}Consider the family (25) in Table 2 of \cite{MO}, of a special family of abelian covers of $\P^1$ of genus 4 with Galois 
 group $(\Z/3)^2$ ramified over 4 points. As this is a special family, one has by Lemma~\ref{monodromydecomp} that $M^{der}=\mon^0(\V)$ for $M$ the generic Mumford-Tate group. In this case the monodromy
 group is isomorphic to $\su(1,1)$.  By Remark~\ref{monodromyeigspace}, the non-compect factor $\su(1,1)$ corresponds to the eigenspace of $(1,2)$. 
  Note the natural projection
 $p_1\colon \mon^{0,ad}(\V)\to \mon^{0,ad}(\mathcal{L}_{(1,2)})$
 gives that the group $\mon^{0,ad}(\V)$ is the direct product of the kernel of this morphism and a 
 semisimple group isomorphic to $\mon^{0,ad}(\mathcal{L}_{(1,2)})$ (if $H:= H_1\times\cdots \times H_n$ 
 is a product of connected simple groups and $B$ is a normal connected subgroup of $H$, then $B=G_1\times\cdots \times G_r\times \{0\}\times \cdots \times \{0\}$ for some $r\leq n$ for a suitable numbering).
 This implies that in particular the group $\mon^{0,ad}(\V)$ decomposes as $\displaystyle \prod_ j \mon^{0,ad}(\mathcal{L}_j)$ of all eigenspaces, for if this is not the case, then according to the above 
 direct product decomposition, there is another
 $j^{\prime}$ such that $\mon^{0,ad}(\mathcal{L}_{(1,2)})$ is not contained in the kernel of the natural projection onto $\mon^{0,ad}(\mathcal{L}_{j^{\prime}})$. Note that
 every simple direct factor of $\mon^{0,ad}(\V)$ projects isomorphically to some $\mon^{0,ad}(\mathcal{L}_{j^{\prime}})$. However since 
 any other $\mon^{0,ad}(\mathcal{L}_{j^{\prime}})$ for $j^{\prime}\neq (1,2)$ is trivial and there can not exist more than one $\psu(1,1)$ in the decomposition
 (as this will imply that the family ist not special, see the next remark) this can not be the case. In particular this decomposition implies that $M^{ad}\cong \psu(1,1)$.
 The group $\su(1,1)$ acts on $\P^1_{\C}$ and its maximal torus $T$ is the subgroup of diagonal matrices so that $\su(1,1)/T$
 is isomorphic to the open unit disk which is therefore the homogeneous space for this group and hence $\delta(\psu(1,1))=1$.

 As another example, the generic Mumford-Tate group of the universal family of hyperelliptic curves of degree $m$ is isomorphic to $\gsp(\V_{m/2},\Q_{m/2})$, see \cite{R}, \S 5.5. 
 \end{example}

 \begin{remark} \label{largenotshimura}
 Assume that $C\rightarrow T$ is a family of
 curves and let $M$ be the generic Mumford-Tate group of this
 family. Recall from Construction~\ref{smallestshimura}, that there is a natural
 Shimura variety $S_{f}=\sh(M,C)$ associated to $M$ (which is a
 reductive group) and the dimension of $S_{f}$ only depends on
 $M^{ad}_{\mathbb{R}}$. The Shimura datum comes from the Hodge
 structures of the fibers in the family. This Shimura variety is
 the smallest Shimura subvariety in $A_{g}$ which contains $Z$. Our
 purpose is to show that for families of abelian covers with a
 large $s$, the moduli variety $Z$ is not a Shimura subvariety.
 This is equivalent to $\sum\delta(Q_{i})>s-3$. Here
 $\delta(Q_{i})$ is as in Construction~\ref{smallestshimura}. Lemma~\ref{monodromyeigspace} enables
 us to compute the connected algebraic monodromy group
 $\mon^{0}(\mathcal{L}_{i})$ of an eigenspace $\mathcal{L}_{i}$.
 According to Remark~\ref{monodromydecomp} above, if $M^{ad}_{\mathbb{R}}=\prod_{1}^{l} Q_{i}$ is the decomposition
 into $\mathbb{R}$-simple factors, then $\mon^{0,ad}_{\mathbb{R}}=\prod_{1}^{l} Q_{i}$ as well. We need to find eigenspaces
 $\mathcal{L}_{j_i}$ of distinct types $\{a_i,b_i\}$
 with $a_i$ and $b_i$ large enough in the sense described below.  Concretely, if we can find eigenspaces $\mathcal{L}_{j_i}$ as above, 
 then these eigenspaces, being of different types $\{a_i,b_i\}$, give rise to non-isomorphic $Q_i=\mon^{0,ad}_{\mathbb{R}}(\mathcal{L}_{j_i})=\psu(a_i,b_i)$ for
 $i\in K$ in the above decomposition of $\mon^{0,ad}_{\mathbb{R}}$ and if for
 $\delta(\mathcal{L}_{j_i})=\delta(\mon^{0,ad}_{\mathbb{R}}(\mathcal{L}_{j_i}))$, we have that 
 $\sum \delta(\mathcal{L}_{j_i})>s-3$ (this is what we mean by \emph{large enough} in the above, note that $\delta(\mathcal{L}_{j_i})$ depends in our 
 examples only on $a_i$ and $b_i$; see Construction ~\ref{smallestshimura}), then
 $\dim S_f\geq \sum \delta(\mathcal{L}_{j_i})>s-3$.  
 \end{remark}

 The following lemma shows that if $s$ is large, then the candidates for Shimura families are those that
contain zeros in each row.

\begin{lemma} \label{zeroinrow}
 Consider a family $f\colon C\to T$ of abelian covers of $\P^1$ with the associated $m\times s$ matrix $A$ as in section 2.  Note that each row corresponds to a family of cyclic covers of the projective line. 
 Suppose that there is one row of $A$ whose entries are all non-zero. If the family of cyclic coverings arising from this row
 is not a special family (of cyclic covers) then the family $f\colon C\to T$ is not a special family too. If the family of cyclic coverings arising from the row is special, then $s\leq 6$.
\end{lemma}
\begin{proof}
 If this family of cyclic covers given by the mentioned row is not a Shimura family
 then by the above notations and observations there are
 $\mathbb{R}$-simple factors $Q_{i}$ in the decomposition of
 $M^{ad}_{1, \mathbb{R}}$ such that $\sum\delta(Q_{i})>s-3$. Here
 $M_{1}$ is the generic Mumford-Tate group associated to the
 family of cyclic covers given by the mentioned row. However, note that the row gives a sub-variation of Hodge structures of the 
 family and we know from \cite{VZ} that $M_{1}$ is then a quotient of $M$, the generic
 Mumford-Tate group of our family $C/T$, and therefore
 the factors $Q_{i}$ also occur in the decomposition of $M^{ad}$
 (note that $M^{ad}$ is a semi-simple group with trivial center)
 and so $\dim S_{f}\geq \sum\delta(Q_{i})>s-3$, i.e., the family is
 not a Shimura family. On the other hand, by results of \cite{M}, if a family of cyclic covers of $\mathbb{P}^{1}$ gives rise
 to a Shimura subvariety, then $s\leq 6$.
\end{proof}
\section{Reduction mod $p$}
Let $f:C\rightarrow T$ be a family of smooth projective
curves with an irreducible base scheme $T$. We denote the sheaf
of relative differentials with $\omega_{C/T}$ and the Hodge
bundle $\mathbb{E}=\mathbb{E}(C/T)=f_{*}\omega_{C/T}$. Consider
the Kodaira-Spencer map $\kappa: \sym^{2}(\mathbb{E})\rightarrow
\Omega^{1}_{T}$ (usually the dual of this map is defined as the Kodaira-spencer map, however since we mainly need this map, rather than
the original one, we name it the Kodaira-spencer map). The
multiplication of forms $\mult:\sym^{2}(\mathbb{E})\rightarrow f_{*}(\omega^{\otimes2}_{C/T})$
induces the sheaf
$\mathcal{K}=\ker(\mult)=\ker(\sym^{2}(\mathbb{E})\rightarrow f_{*}(\omega^{\otimes2}_{C/T}))$.
If the fibers are not hyperelliptic, $\mult$ is surjective and $\mathcal{K}$ is dual to $\coker(\mult)$.
 
 \

 The curves and their families which we introduced up to now 
 have only been defined over $\mathbb{C}$. The use of characteristic $p$ tools in studying the Shimura subvarieties 
 of $A_g$ was initiated in \cite{DN} by applying constructions of Dwork and Ogus and later was also used in \cite{M} and \cite{MZ}. In order to be able to use
 these characteristic $p$ tools, we need to set the scene in such a way that the reduction mod $p$ makes sense and in order to do this we need first the following
 fundamental definiton.
 
\begin{definition}
Let $k$ be an algebraically closed field of characteristic $p>0$. Let $A$ be an abelian variety over $k$. 
Write $A[p]=\ker (p\cdotp 1_{A}:A\to A)$. The abelian variety $A$ is called \emph{ordinary} if $A[p](k)\cong (\Z/p)^g$ for $g=\dim(A)$. 
A smooth curve $X$ over an algebraically closed field of characteristic $p>0$ is called ordinary if its Jacobian is an ordinary abelian
variety.
\end{definition}
For example the Fermat curve of degree 5 is an ordinary curve of genus 6 in characteristic $p\equiv 1$ (mod 5), see \cite{DO}, p. 130. Many other examples of ordinary curves
have been constructed by Oort and Sekiguchi in \cite{OS} using Galois covers of $\P^1$. \\

 Let $f:C\rightarrow T$ be a family of abelian covers as in section
 2. Let $R=\mathbb{Z}[1/N,u]/\Phi_N$, where $\Phi_N$ is the $N$th
 cyclotomic polynomial. Note that $R$ can be embedded into
 $\mathbb{C}$ by sending the image of $u$ to $\exp (2\pi i/N)$. We
 consider $T\subset (\mathbb{A}^1_R)^s$ as the complement of the
 big diagonals, i.e., as the $R$-scheme of ordered $s$-tuples of
 distinct points in $\mathbb{A}^1_R$. For a prime number $p$, we
 denote by $\wp$ a prime ideal of $R$ lying above $p$ (meaning that $\wp\cap \Z=p\Z$). One can choose a
 prime number $p\equiv 1$ (mod $N$) and an open subset $U$ of
 $T\otimes \mathbb{F}_{p}\cong T\otimes_R R/\wp$ such that for all
 $t\in U$, the fibers are ordinary curves in characteristic $p$. For such $p$ and $U$, consider the restricted family $C_{U}\rightarrow U$.
 The abelian group $G$ also acts on the sheaves
 $\mathcal{L}(C_{U}/U)$ and gives the eigensheaf decomposition
 $\mathcal{L}(C_{U}/U)=\oplus_{n\in G}\mathcal{L}_{(n)}$. The same
 is true for $\mathbb{E}_{U}=\mathbb{E}(C_{U}/U)$ and
 $\mathcal{K}_{U}=\mathcal{K}(C_{U}/U)$.

 \

 From now on we just work with the restricted family $C_{U}/U$
 whose fibers are all ordinary instead of $C/T$ and denote it
 simply as $C/U$. Let $F_{U}:U\rightarrow U$ denote the absolute Frobenius map. In the following result, the superscript $G$
 denotes the $G$-invariant sections. 
\begin{proposition} \label{exactsequencevanish}

If the family of abelian covers gives rise to a
Shimura subvariety in $A_{g}$, then the map
\[F^{*}_{U}\mathcal{K}^G\hookrightarrow
F^{*}_{U}\sym^{2}(\mathbb{E}_{U})^G\xrightarrow{\sym^{2}(\gamma)}
\sym^{2}(\mathbb{E}_{U})^G\xrightarrow{\mult^G}
f_{*}(\omega_{C/U}^{\otimes2})^G\]
vanishes identically.
\end{proposition}
\begin{proof}
See \cite{MZ}, Proposition 4.3 (after \cite{M}, Prop. 5.8). 

\end{proof}

The effect of the Frobenius map on the eigenspaces can be realized by the following lemma. Note that the $\alpha_i$ below are as in 
Remark~\ref{eigenspacedim} .

\begin{lemma} \label{HasseWitt}
With notation as above, the eigenspaces $H^{1}(Y,\mathcal{O}_{Y})_{\chi}$ are stable under the Hasse-Witt map and there is a basis
$(\xi_{a,i})_{i}$ on each eigenspace in which the $(i,j)$ entry of matrix is given by the formula:
\[\sum_{\sum l_{i}=\Upsilon}\binom{q.[-\alpha_{1}]_{N}}{l_{1}}...\binom{q.[-\alpha_{s}]_{N}}{l_{s}}z_{1}^{l_{1}}...z_{s}^{l_{s}},\]
where $\Upsilon= (d_{n}-i+1)(p-1)+(i-j)$ and
$\binom{a}{b}=\frac{a!}{b!(a-b)!}$.
\end{lemma}

\begin{proof}

\cite{MZ}, Lemma 5.1.
\end{proof}

\section{Main results}

In this section, after providing the defintions and notations which we need, we prove our results.
Let us first formulate a condition which we will need later\\

(*) There exists $n\in G$ such that $\{d_n, d_{-n}\}\neq \{0,s-2\}$ and $d_n+d_{-n}\geq s-2$.\\

It is straightforward to see that this condition is not vacuous and indeed it is satisfied for infinitely many
families of abelian covers, so that Theorem~\ref{mainthm1} excludes infinitely many non-trivial examples. However,
there do exist families which do not satisfy this condition, see family $III^*$ below. 

\begin{remark} \label{irreduciblecover}
 
In this section we will will work only with families of \emph{irreducible} abelian covers of
$\mathbb{P}^{1}$, i.e., the fibers of the family are irreducible curves. For cyclic covers this implies that the single row of 
the associated matrix is not annihilated by a non-zero element of $\mathbb{Z}/N\mathbb{Z}$. More generally, for abelian covers of $\mathbb{P}^{1}$ this implies
that the rows of the associated matrix are linearly independent over $\mathbb{Z}/N\mathbb{Z}$.
\end{remark}

We also remark that if $\dim S(G)=s-3$, then, as explained earlier, the family is a special family, hence the condition $\dim S(G)>s-3$ is actually required
in the next theorem.
\begin{theorem} \label{mainthm1}
Let $C\to T$ be a family of abelian covers of the line such that $\dim S(G)>s-3$. If this family satisfies condition (*), then it does not 
give rise to a special subvariety of $A_g$. 
\end{theorem}

\begin{proof}
Assume on the contrary that $C/T$ is a special family and take $n\in G$ as in condition (*). 
Note that $d_nd_{-n}\geq s-3$ with equality
if and only if $\{d_n,d_{-n}\}=\{1,s-3\}$. Therefore if
$\{d_n,d_{-n}\}\neq \{1,s-3\}$, then $\dim S_f>s-3$ and we are
done. Hence, we may assume that $\{d_n,d_{-n}\}=\{1,s-3\}$. Next,
we claim that there exists another $n^{\prime}\in G$ such that
$\{d_{n^{\prime}},d_{-n^{\prime}}\}=\{1,s-3\}$. Suppose that this not
the case. Observe that if for every $n^{\prime}\in G$,
$\{d_n,d_{-n}\}=\{0,a\}$ for some $a$, then Lemma~\ref{dimS(G)} gives that
$\dim S(G)=d_nd_{-n}=s-3$ which is against our assumptions. Hence
there exists a $n^{\prime}\in G$ such that $d_{n^{\prime}}\neq 0$
and $d_{-n^{\prime}}\neq 0$. If
$\{d_{n^{\prime}},d_{-n^{\prime}}\}\neq \{1,s-3\}$, then we have
an eigenspace of new type and consequently, $\dim S_f\geq
d_nd_{-n}+d_{n^{\prime}}d_{-n^{\prime}}>s-3$. So we may and do
assume that $\{d_{n^{\prime}},d_{-n^{\prime}}\}= \{1,s-3\}$. In this case both eigenspaces correspond to the same 
factor in the decomposition of $M^{ad}_{\mathbb{R}}$, see Construction~\ref{smallestshimura}. 
Let $p$ and $U$ be as in \S 2.1. So $p$ is a prime number such that $p\equiv 1 \text{ mod } N$ and set 
$q=\frac{p-1}{N}$. For these choices, consider the Hasse-Witt map
$\gamma_{(n)}:F_U^*\mathbb{E}_{U,(n)}\to \mathbb{E}_{U,(n)}$ and
let $\Gamma\in \gl_{s-3}(\mathcal{O}_U)$ with respect to the basis
$\omega_{n,\nu}$ introduced earlier in Remark~\ref{eigenspacedim}. Lemma~\ref{HasseWitt} (actually its dual version) gives a
description of the matrix $A=A_n$. We also denote
by $\gamma_{(n^{\prime})}$ and $A^{\prime}$ the corresponding
Hasse-Witt map and matrix respectively for $n^{\prime}$. We may, without loss of
generality, assume that $d_{-n}=d_{-n^{\prime}}=1$ and
$d_{n}=d_{n^{\prime}}=s-3$. Hence $A_{-n}$ and $A_{-n^{\prime}}$
are $1\times 1$ matrices, i.e., can be considered as sections of
$\mathcal{O}_U^*$ which we denote by $a, a^{\prime}$ respectively.
For each $\tau\in \{0,\cdots, s-4\}$, let $\varphi_{\tau}\in
\Gamma(U,\mathcal{K}^G)$ be defined by
\[\varphi_{\tau}:=\omega_{-n,0}\otimes \omega_{n,\tau}-\omega_{-n^{\prime},0}\otimes \omega_{n^{\prime},\tau}\]
Note that since
$\omega_{-n,0}.\omega_{n,\nu}=\omega_{-n^{\prime},0}.
\omega_{n^{\prime},\nu}$ as sections of $f_{*}(\omega^{\otimes
2})$, it follows that the image of $\varphi_{\tau}$
under $\sym^2(\gamma)$ is equal to
\[\displaystyle \sum_{\nu=0}^{s-4}(a\cdotp\Gamma_{\nu, \tau}-a^{\prime}\Gamma_{\nu, \tau})(\omega_{-n,0}\cdotp\omega_{n,\nu}).\]
But the sections $\omega_{-n,0}\cdotp\omega_{n,\nu}$ are linearly
independent for $\nu\in \{0,\cdots, s-4\}$, so Proposition~\ref{exactsequencevanish}
gives that $a\Gamma_{\nu, \tau}-a^{\prime}\Gamma_{\nu, \tau}=0$
for all $\tau, \nu \in \{0,\cdots, s-4\}$. This can be rewritten as 
$a^{-1}A_{\nu, \tau}=a^{\prime -1}A^{\prime}_{\nu, \tau}$. By examining the powers of various indeterminates in both sides of this equation, 
we show that this equality can not hold. Assume the contrary. Pick two distinct 
$i,j\in \{1,\cdots, s\}$ and set $I=\{1,\cdots, s\}\setminus \{i,j\}$. For $h= 1,2$, define
$r_n(h)$ similarly as in \cite{MZ}, proof of Theorem 6.2, to be the largest integer $r$ such that $A_{h,h}|_{t_i=0}$ is divisible by $t^{r}_{j}$. Similarly, let
$r_{-n}$ be the largest integer $r$ such that $a^{-1}|_{t_i=0}$ is divisible by $t^{r}_{j}$. Let $n\cdotp A=(\alpha_1,\cdots, \alpha_s)$ be as in Remark~\ref{eigenspacedim}. By the formulas
for $a^{-1}$ and the matrix $A$ given in Lemma~\ref{HasseWitt}, we find
\[r_{-n}=\displaystyle \max \{0,(p-1)\cdotp q\sum_{i\in I}[\alpha_i]_N\}\]
\[r_{n}(1)=\displaystyle \max \{0,(s-3)(p-1)\cdotp q\sum_{i\in I}[-\alpha_i]_N\}.\]
Note that $r_{n}(2)=0$.
With the above definitions, $u_n(h)=r_{-n}+r_n(h)$ is the largest
integer $u$ such that $(a^{-1}A_{h,h})|_{t_i=0}$ is divisible by
$t_{j}^u$. The facts that $d_{-n}=1$ and $d_{n}=s-3$ show that
$\sum [\alpha_i]_N=2N$ and $[-\alpha_i]_N=N-[\alpha_i]_N$ for
every $i$. By these relations one gets
\[u_n(1)=\displaystyle \max \{(p-1)-q\sum_{i\notin I} [\alpha_i]_N, (p-1)-q\sum_{i\in I} [\alpha_i]_N\}
=|(p-1)-q\sum_{i\notin I}
[\alpha_i]_N|=q\cdotp|N-[\alpha_k]+[\alpha_{\lambda}]|,\]
\[u_n(2)=\displaystyle \max \{0, (p-1)-q\sum_{i\notin I} [\alpha_i]_N\}=\max \{0, (p-1)-q\sum_{i\in I} [\alpha_i]_N\}=
q\cdotp\max \{0, [\alpha_k]+[\alpha_{\lambda}]-N\}\]

Analogously, we can define the above notions for $\pm n^{\prime}$
which we represent by $u^{\prime}(1), u^{\prime}(2)$. The above
relations give that $u(1)=u^{\prime}(1)$, $u(2)=u^{\prime}(2)$.
This implies that for all $k,\lambda$, we have
$[\alpha_k]_N+[\alpha_{\lambda}]_N=[\alpha^{\prime}_k]_N+[\alpha^{\prime}_{\lambda}]_N$.
From this one gets that $[\alpha_i]_N=[\alpha^{\prime}_i]_N$ for every $i\in \{1,\cdots, s\}$. But
this implies that the two rows corresponding to $n$ and
$n^{\prime}$ are equal and in particular linearly dependent. This
is in contradiction with our assumptions by Remark~\ref{irreduciblecover}.
\end{proof}

In \cite{MZ} we have used the above reduction method in order to prove
that further examples of Shimura families of abelian covers of
$\mathbb{P}^{1}$ do not exist for $s=4$. However, note that for $s>5$, condition (*) may not hold.
i.e., it could very well be that $d_n+d_{-n}< s-2$ for all $n\in G$. Example $III^{*}$ below (see Theorem~\ref{mainthm2}) is
the simplest example of such a family. In the following, we restrict our attention to
families with Galois groups of the form $\mathbb{Z}/n\mathbb{Z}
\times \mathbb{Z}/m\mathbb{Z}$ i.e. that the matrix $A$ is a
$2\times s$ matrix.

We are going to show that for $s=5$, there are eigenspaces
$\mathcal{L}_{i}$ of types $(a_{i},b_{i})$ with
$\{a_{i},b_{i}\}\neq\{a_{j},b_{j}\}$ for $i\neq j$ and such that
$\sum \delta(\mathcal{L}_{i})>s-3 $. Then by the above Remark~\ref{zeroinrow},
we conclude that $\dim S_{f}>s-3$. Note that this property is
far from being true for the families of cyclic covers of
$\mathbb{P}^{1}$. For those families it can happen that all of the
eigenspaces are either unitary (i.e. $a_{i}=0$ or $b_{i}=0$) or of
the same type. Take for example the family $(11,(1,1,1,1,7))$. In
this case all of eigenspaces are either of type $(3,0)$ (or
$(0,3)$) and hence unitary, or of type $(1,2)$. On the other hand, if 
this cyclic family is a Shimura family it must be one of the families in \cite{M} and
therefore, $N$ is one the $10$ numbers in table $1$ of \cite{M}. Of course this leaves only finitely many possibilities to
investigate. So, if in one of the rows (or the tuples associated to eigenspaces by virtue of Remark~\ref{eigenspacedim}) all of the entries are
non-zero, according to the table in \cite{M}, $N=3,4,5,6$ and we will
exclude these in what follows. In the case that there are $0$ entries in the
rows, consider the following families:
 $I)\begin{pmatrix} a_{1}&a_{2}&a_{3}&0&0\\
0&0&0&b_{2}&b_{3} \end{pmatrix}$, $II)\begin{pmatrix} a_{1}&a_{2}&a_{3}&0&0\\
0&0&b_{1}&b_{2}&b_{3}\\
 \end{pmatrix}$, $III)\begin{pmatrix}
a_{1}&a_{2}&a_{3}&a_{4}&0\\
0&0&b_{1}&b_{2}&b_{3}
  \end{pmatrix}$ or $IV)\begin{pmatrix}
a_{1}&a_{2}&a_{3}&a_{4}&0\\
0&0&0&b_{1}&b_{2}
  \end{pmatrix}$.

  \

Also consider the families $III^{*})\begin{pmatrix} 2&1&2&1&0\\
0&0&1&1&1\\
\end{pmatrix}$ with $N=3$ and
\

$III^{**})\begin{pmatrix} 2&2&3&1&0\\
0&0&1&1&2\\
\end{pmatrix}$ with $N=4$, which are special
cases of families $III$. The following theorem
can serve as an example of how the above tools based on monodromy can be
applied to exclude Special subvarieties of smaller dimensions from the Torelli locus.

\begin{theorem} \label{mainthm2} A 2-dimensional family $C\rightarrow T$ of irreducible abelian 
covers of $\mathbb{P}^{1}$ of the above form does not give rise to a Shimura variety in $A_{g}$ except
possibly the families $III^{*}$ and $III^{**}$. In particular, the set of CM fibers in these families is not dense. 
\end{theorem}
\begin{proof}  Let us show that families II are not special families. In this case, the
eigenspace associated to the element $(1,1)$ is given by the associated tuple
$(a_{1},a_{2},a_{3}+b_{1},b_{2},b_{3})$ by Remark~\ref{eigenspacedim}, so we must have
$a_{3}+b_{1}=0$, otherwise $N=3,4,5,6$ by the above argument just before the theorem, which we will exclude below. Likewise
$a_{3}-b_{1}=0$ and so we have that $a_{3}=b_{1}=\frac{N}{2}$.
Consider the eigenspace associated to the element $(2,1)$ given by
the cyclic cover $(2a_{1},2a_{2}, \frac{N}{2}, b_{2}, b_{3})$
since none of $a_{1}$ and $a_{2}$ is zero, we have also that
$2a_{1}\neq 0$ and $2a_{2}\neq 0 $ (otherwise for example $a_{1}=\frac{N}{2}$ but we may assume that $\sum
a_{i}=N$, otherwise we can replace $a_{i}$ with $-a_{i}$ and we
get an isomorphic cover for which $\sum a_{i}=N$. This forces $a_2=0$, a contradiction!). By what we said
earlier, this implies that $N=4,6$ which we have to exclude now.
For $N=4$, the only possible families are $\begin{pmatrix}
  1 & 1 & 2 & 0 & 0 \\
  0 & 0 & 2 & 1 & 1 \\
\end{pmatrix}$
 and $\begin{pmatrix}
  2 & 1 & 1 & 0 & 0 \\
  0 & 0 & 1 & 1 & 2 \\
\end{pmatrix}$ and $\begin{pmatrix}
1&1&2&0&0\\
0&0&1&1&2\\
\end{pmatrix}$. All of the $3$ families are not special as in each case, there is 
an eigenspace of type $(1,2)$ and another one of type $(1,1)$ (For example in the first case, the eigenspace
$\mathcal{L}_{(1,2)}$ has type $(1,2)$ and the eigenspace
$\mathcal{L}_{(1,3)}$ has type $(1,1)$). This shows by the above
remarks that $\dim S_{f}=\sum \delta(Q_{i}) \geq 3$. Therefore $Z
\neq S_{f}$ and the family is not special. In the same way one
sees easily that the other two families are not special too. Finally
by the same method one can conclude that for $N=3,4,5,6$, there does not
exist any special family of the form $II$. For example for $N=3$ there is only one family,
namely the
family $\begin{pmatrix} 1&1&1&0&0\\
0&0&1&1&1\\
\end{pmatrix}$ which is not special again because there is an eigenspace of
type $(1,2)$ and another one of type $(1,1)$ which forces
$\dim S_{f} \geq 3$. Next, we trun to families $III$. Again by Remark~\ref{eigenspacedim}, the eigenspace associated to
the element $(1,1)$ is given by $(a_{1},a_{2},a_{3}+b_{1},a_{4}+b_{2},b_{3})$. Hence either 
$a_{3}+b_{1}=0$ or $a_{4}+b_{2}=0$. Similarly, either $a_{3}-b_{1}=0$ or $a_{4}-b_{2}=0$. We first assume 
that $a_{3}+b_{1}=0$ and $a_{4}-b_{2}=0$. Now consider the eigenspace of $(2,1)$ which is associated to the 
tuple $(2a_{1},2a_{2},[-b_{1}],3b_{2},b_{3})$. Hence we get $2a_{1}=0$ or $2a_{2}=0$ or $3b_{2}=0$. If $2a_{1}=0$ (the case $2a_{2}=0$ is analogous) 
or equivalently $a_{1}=\frac{N}{2}$, the eigenspace of $(1,2)$ has the associated tuple $(\frac{N}{2},a_{2},b_{1},3b_{2},2b_{3})$. So 
either $3b_{2}=0$ or $2b_{3}=0$. If $2b_{3}=0$ then $b_{3}=\frac{N}{2}$. Consider the element $(1,3)$ whose eigenspace
is given by the associated tuple $(\frac{N}{2},a_{2},2b_{1},4b_{2},\frac{N}{2})$. The case $2b_{1}=0$ can not happen because already $b_3=\frac{N}{2}$ 
and $\sum b_i=N$. So we must only have $4b_{2}=0$. However again the only possible choice is $b_2=\frac{N}{4}$ and this gives the family isomorphic to
$\begin{pmatrix}
  2 & 2 & 3 & 1 & 0 \\
  0 & 0 & 1 & 1 & 2 \\
\end{pmatrix}$. Now assume that $3b_{2}=0$. Hence $b_2=\frac{N}{3}$ or $b_2=\frac{2N}{3}$. Again consider the eigenspace associated to 
$(1,3)$ which is given by the tuple $(\frac{N}{2},a_{2},2b_{1},\frac{N}{3},3b_3)$. It follows that $b_3=\frac{N}{3}$, which does not give rise to an irreducible family.
Similarly, the case that $3b_2=0$ (and $2a_1\neq 0$ and $2a_2\neq 0$) gives the family $III^*$ (by 
considering the eigenspaces associated to $(1,3)$ and $(3,1)$). It remains to treat the 
case that $a_{3}+b_{1}=0$ and $a_{3}-b_{1}=0$ which implies that $a_{3}=b_{1}=\frac{N}{2}$. We take the eigenspace of $(2,1)$ whose associated tuple is 
$(2a_{1},2a_{2},\frac{N}{2},3b_{2},b_{3})$. We first assume that $2a_{1}\neq 0,2a_{2}\neq 0$ and consider the  eigenspace of $(1,2)$ which gives rise to 
$(a_{1},a_{2},[-3a_4],3b_{2},2b_{3})$. The case $2b_{3}=0$ is not possible because already we have $b_1=\frac{N}{2}$. So $3a_4=0$ which gives $a_4=\frac{N}{3}$. 
But this implies that $a_4=b_2=\frac{N}{3}$ which is against our assumption. Finally suppose for example that $2a_{1}=0$ ($2a_{2}=0$ is analogous) and take the eigenspace of $(1,2)$ 
corresponding to $(\frac{N}{2},a_{2},\frac{N}{2},a_4+2b_2,2b_{3})$. As $2b_3=0$ is not possible, it follows that
$a_4=[-2b_2]$. Take the eigenspace of $(1,4)$ giving the tuple $(\frac{N}{2},a_{2},\frac{N}{2},2b_2,4b_{3})$. As $2b_2\neq 0$, it follows that $4b_{3}=0$ and the only 
possible case is $b_{3}=\frac{N}{4}$ which implies also that $b_{2}=\frac{N}{4}$. But this does not give rise to an irreducible family. Similarly, one can verify by the above 
method that families $I$ and $IV$ are also not special families. 
\end{proof}

\end{document}